\documentclass{article}

\usepackage[utf8]{inputenc}
\usepackage[T1]{fontenc}
\usepackage[a4paper]{geometry}
\usepackage{amsmath,stackengine,scalerel}
\usepackage{amsthm} 
\usepackage{mathtools}
\usepackage{graphicx}
\usepackage[english]{babel}
\usepackage{hyperref}
\usepackage{amsfonts} 
\usepackage{mathtools}
\usepackage{comment}
\usepackage{appendix}
\usepackage{color} 
\usepackage{lineno} 
\usepackage{float} 
\usepackage{caption} 

\newtheorem{thm}{Theorem}[section]
\newtheorem{lemma}{Lemma}[section]
\newtheorem{defin}{Definition}[section]

\newcommand{\numberset}{\mathbb}
 
\newcommand{\R}{\numberset{R}} 
\newcommand{\nl}[2]{\|#1\|_{L^2(#2)}} 
\newcommand{\nlstar}[3]{\|#1\|_{L^{2^{#2}}(#3)}} 
\newcommand{\nlnab}[2]{\|\nabla #1\|_{{L^2(#2)}^d}} 
\newcommand{\nh}[2]{\|#1\|_{H^1(#2)}} 
\newcommand{\h}[1]{\widehat{#1}}
\newcommand{\pt}{\partial_t}
\newcommand{\p}{\partial}

\newcommand{\ithx}{\int_{0}^{T-h}\int_{\Omega}}

\title{Existence of a global weak solution for a reaction-diffusion problem with membrane conditions}

\author{Giorgia Ciavolella\thanks{Sorbonne Universit{\'e}, CNRS, Universit\'{e} de Paris, Inria, Laboratoire Jacques-Louis Lions UMR7598, 75005 Paris, France. Emails : giorgia.ciavolella@sorbonne-universite.fr, Benoit.Perthame@sorbonne-universite.fr}
	\and
	Beno\^ \i t Perthame\footnotemark[1]
}

\date{\today}

\begin{document} 
		
	\maketitle
	
	\begin{abstract} 
		Several problems, issued from physics, biology or the medical science, lead to parabolic equations set in two sub-domains separated by a membrane with selective permeability to specific molecules. The corresponding boundary conditions, describing the flow through the membrane, are compatible with mass conservation and energy dissipation, and are called the Kedem-Katchalsky conditions. Additionally, in these models, written as reaction-diffusion systems, the reaction terms have a quadratic behaviour. 
		
		M. Pierre and his collaborators have developed a complete $L^1$ theory for reaction-diffusion systems with different diffusion. Here, we adapt this theory to the membrane boundary conditions and prove the existence of weak solutions when the initial data has only $L^1$ regularity using the truncation method for the nonlinearities. In particular, we establish several estimates as the $W^{1,1}$ regularity of the solutions. Also, a crucial step is to adapt the fundamental $L^2$ (space, time) integrability lemma to our situation. 
		
	\end{abstract} 
	\vskip .7cm
	
	\noindent{\makebox[1in]\hrulefill}\newline

	\section*{Introduction}
	
	We analyse the existence of a global weak solution for a reaction-diffusion problem  of $m$ species which diffuse through a permeable membrane. This kind of problem is described by the so-called Kedem-Katchalsky conditions \cite{KK} and has been used in mathematical biology recently. They can describe transport of molecules through the cell/nucleus membrane \cite{serafini}, the flux of cancer cells through thin interfaces \cite{giverso} or solutes absorption processes through the arterial wall \cite{quarteroni}.
	
	To describe the model, we consider, as depicted in Fig.~\ref{fig1}, an inner transverse $C^1$ membrane~$\Gamma$ separating a domain~$\Omega$ in two connected sub-domains  $\Omega^1$ and $ \Omega^2 $, 
	\[
	\Omega~=~\Omega^1 \,\cup\, \Omega^2 \subset \R^d, \quad d\geq 2, \qquad \qquad \Gamma=\partial \Omega^1 \cap \partial\Omega^2.
	\] 
	We assume $\Omega^1$ and $\Omega^2$ to be piecewise $C^1$ domains. In order to set boundary conditions, we introduce $\Gamma^1=\p \Omega^1\setminus \Gamma$ and $\Gamma^2=\p\Omega^2\setminus \Gamma$. We assume that  $\Gamma^1$ and $\Gamma^2$ are non-empty. We could also consider a different geometry such that $\Omega^1$ includes $\Omega^2$ and the membrane becomes the boundary of the inner domain (see for example \cite{BCF,LiWang,wang}). In contrast, the biological situation that we  analyse is presented in Fig.~\ref{fig1} and that is why we leave open the problem with an inner domain.
	\captionsetup[figure]{labelfont=bf,textfont={it}}
	\begin{figure}[H] 	
		\begin{center}
			\includegraphics[scale=0.17]{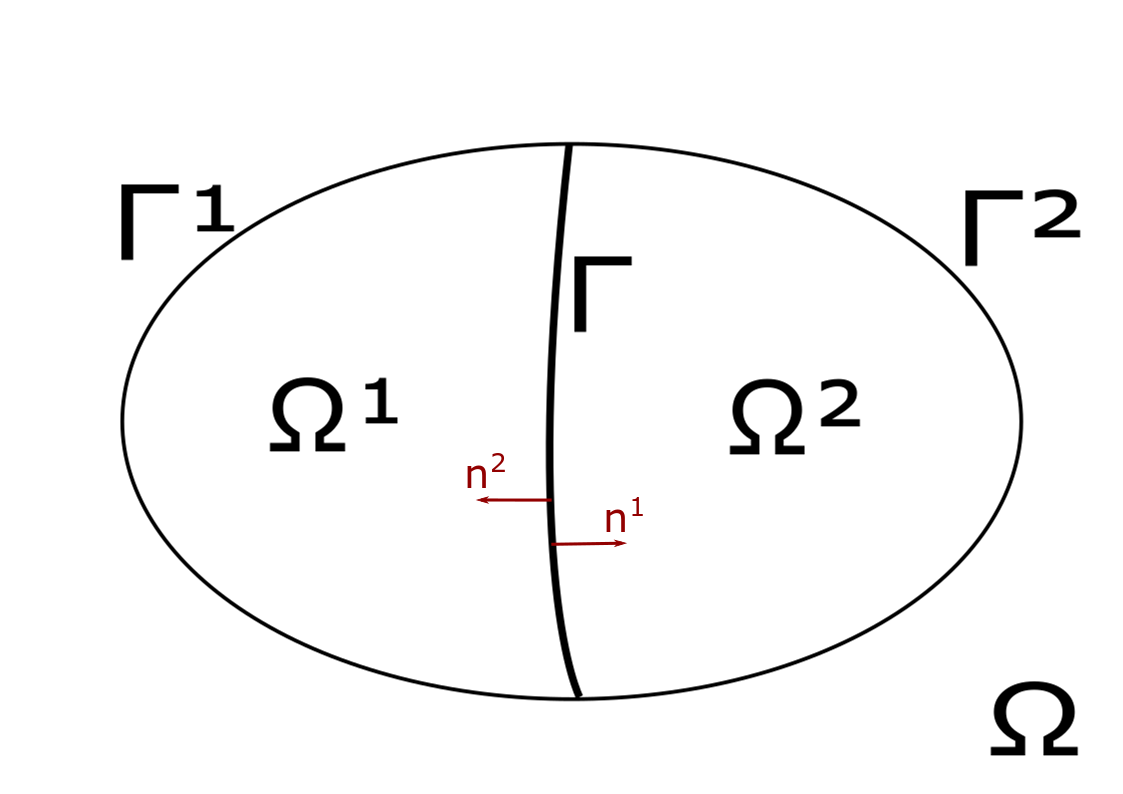}
			\caption{Example of spatial domain $\Omega $ with an inner transverse membrane $\Gamma$ which decomposes $\Omega$ in open sets $\Omega^1$ and $\Omega^2$. The figure also shows the outward normals to the membrane.}
			\label{fig1}
		\end{center}
	\end{figure}

	Ignoring a possible drift, the diffusion through the membrane is described by the  system, for species $i=1,..,m$,
	\begin{equation}\label{model}
	\left\{
	\begin{array}{ll}
	\partial_t u_i-D_i\Delta u_i=f_i(u_1,...,u_m),  \qquad& \mbox{in}\; Q_T:=(0,T)\times \Omega,\\[1ex]
	u_i=0,& \mbox{in}\; \Sigma_T:=(0,T)\times(\Gamma^1\cup\Gamma^2),\\[1ex]
	\partial_{\boldsymbol{n}^1} u^1_i=\partial_{\boldsymbol{n}^1} u^2_i=k_i(u_i^2-u_i^1), & \mbox{in}\; \Sigma_{T,\Gamma}:=(0,T)\times \Gamma,\\[1ex]
	u_i(0,x)=u_{0,i}(x) \geq 0, & \mbox{in}\; \Omega,
	\end{array}
	\right.
	\end{equation}
	in which $\;D_i $ and $k_i$ are positive constants and $\boldsymbol{n}^\lambda$ is the outward normal of the domain $\Omega^\lambda$ for $\lambda=1,2$ such that $\boldsymbol{n}^2=-\boldsymbol{n}^1$. In particular, we use the notation $\partial_{\boldsymbol{n}^1} u^\lambda_i = \nabla u^\lambda_i \cdot \boldsymbol{n}^1$.
	We denote each species density for $i=1,...,m$ with 
	$$
	u_i=\left\{\begin{array}{ll}
	u^1_i,& \mbox{ in } \Omega^1,   \\[1ex]
	u^2_i,& \mbox{ in } \Omega^2, 
	\end{array}\right.
	$$
	since each one lives in both sub-domains $\Omega^\lambda$, for $\lambda=1,2$. There is a jump of $u_i$, $i=1,...,m$ across the membrane $\Gamma$ that we denote by $(u_i^2-u_i^1)=:[u_i]$. More precisely, for $x\in\Gamma$ and for all $i=1,...,m$, we define the trace in Sobolev sense
	    $$
		u_i^1(x) = \lim_{h\rightarrow 0^{-}} u_i(x+h\;\boldsymbol{n}^1(x)), \qquad u_i^2(x) = \lim_{h\rightarrow 0^{-}} u_i(x+h\;\boldsymbol{n}^2(x)).
		$$

	The interest of this system stems from the boundary conditions. In fact, besides standard Dirichlet boundary condition on $\Gamma^\lambda,$ for $\lambda=1,2$, we have used the Kedem-Katchalsky membrane conditions~\cite{KK} on~$\Gamma$. These conditions are made up by two principles: the conservation of mass, which brings to flux continuity, and the dissipation principle such that the $L^2$-norm of the solution is decreasing in time. This last property gives us that  the flux is proportional to  the jump of the density through the membrane with proportionality coefficient $k_i$,  the membrane permeability constant. 
	These Kedem-Katchalsky interface conditions were introduced in $1961$ in \cite{KK} in a thermodynamic context and they were applied to biological problems only later. 
	In $2002$, Quarteroni \& all. \cite{quarteroni} used these interface conditions in the study of the dynamics of the solute in the vessel and in the arterial wall.
	In $2006$, Calabrò and Zunino \cite{calabro} applied their theoretical results on elliptic partial differential equations to the study of the behavior of a biological model for the transfer of chemicals through thin biological membranes.
	In $2007$, Serafini, in her PhD thesis \cite{serafini}, studied a model of the intracellular signal transduction processes in which molecules freely diffuse and the membrane transport events are allowed.
	In $2010$, Cangiani and Natalini \cite{cangiani} considered models of nuclear transport of molecules such as proteins in living cells taking into account the active transport of molecules along the cytoplasmic microtubules.
	We also find Kedem-Katchalsky conditions in recent works studying tumor invasion such as in the pressure equation in Gallinato \& all. (\cite{gallinato}, $2017$) or in the tumor cell density's equation in Chaplain \& all. (\cite{giverso}, $2019$).
	In \cite{wang} ($2019$), Li \& all. proposed a rigorous derivation of bulk surface models which describe cell polarization and cell division including also transmission conditions.  Let us also mention  an example of transmission condition in electrochemistry: Bathory \& all. (\cite{bulicek}, $2019$) proposed a problem frequently used when modelling the transfer of ions through the interface between two different materials.
	\\
	
	For the applications we have in mind, System~\eqref{model} has mass control, membrane conditions are conservative, and we are interested in developing a theory of weak solutions based on this $L^1$ bound even if the reaction terms are, for instance, quadratic. For usual reaction-diffusion systems, such a theory has been developed in a series of papers initiated by M. Pierre and developed later by several authors. In particular, we extend, to the case of membrane conditions, the method proposed by M. Pierre in \cite{baras, bothe, pierre} and extended by E.-H. Laamri and M. Pierre \cite{lpi}, E.-H. Laamri and B. Perthame \cite{lp}. This method develops a theory to treat high order nonlinearities and low regularity initial data compatible with the natural $L^1$ regularity of solutions.	
	Moreover, we show that for all $i=1,...,m$, $\lambda=1,2$,  $u_i^\lambda\in W^{1,1}(\Omega^\lambda)$ (and even better), but it does not have $L^1$ derivatives in the whole $\Omega$.
	In any case, since $u_i$, $i=1,...,m$ is a Sobolev function in $\Omega^1$ and $\Omega^2$, the trace makes sense  in $\partial \Omega$ and thus the definition of the jumps $[u_i]$, $i=1,...,m$ is meaningful. Finally, we define $\boldsymbol{u}=(u_1,...,u_m)$ the vector solution which is characterized by nonnegative components and, as we will see later on, they are naturally $L^1$ functions but not $L^2$. One of the difficulties of a membrane problem is to derive an  $L^2(Q_T)$ estimate.
	\\ 
	
	In this work, we prove analytical results concerning existence of solutions  and regularity of solutions in the case of the reaction-diffusion systems with  Kedem-Katchalsky conditions~\eqref{model}.	
	The paper is composed of two sections. In Section~\ref{hp}, we introduce the assumptions and our main result about global existence of a weak solution for the Problem~\eqref{model} with related lemmas. We also present a specific example in order to give a more concrete idea of the type of systems of interest for us. In Section~\ref{proof}, we prove this result introducing the approximation model of \eqref{model} (Subsection~\ref{regpb}), proving and applying an a priori $L^2$ estimate on the solution (Subsection~\ref{step2}), proving a theorem about the existence of a supersolution of~\eqref{model} (Subsection~\ref{step3}) and a second one on the existence of a solution (Subsection~\ref{step4}).	
	At the end of this work, the reader can find three Appendices. 
	Appendix~\ref{appendix_reg} and Appendix~\ref{appendix_compact} contain the proof of a regularity and compactness lemma useful in the third step of the proof of our main result. Appendix~\ref{appendix_inequal} provides Sobolev and Poincaré embeddings in the case of membrane conditions and, in general, of non-uniform zero boundary conditions.
	
	\section{Assumptions and main results}
	\label{hp}
	
	\subsection{Assumptions}
	
	We gather several assumptions on the reaction term ${\bf f}({\bf u})= (f_1({\bf u}),...,f_m({\bf u}))$ that are used separately throughout the paper. With some constants $C, C_M$ and $M>0$, we assume that for all $ i=1,...,m$  and for all~$\boldsymbol{u}=~(u_1,...,u_m) \in [0,+\infty)^m$,
	\begin{eqnarray}
	&& |f_i(\boldsymbol{u})|\leq C\Big(1+\sum_{j=1}^{m} u_j^2\Big),   \qquad  \qquad  \qquad   \quad \mbox{(sub-quadratic growth)}, \label{hpF0}
	\\
	&&\sum_{j=1}^{m} f_j(\boldsymbol{u})\leq C\Big(1+\sum_{j=1}^{m} u_j\Big),  \qquad  \qquad  \qquad   \qquad \mbox{(mass control)},  \label{hpF1}
	\\
	&&f_i(u_1,...,u_{i-1},0,u_{i+1},...,u_m)\geq 0, \qquad \qquad \;  \mbox{(quasi-positivity)},   \label{hpF2}
	\\[1ex]
	&&|f_i(\boldsymbol{u})-f_i(\boldsymbol{v})|\leq C_M\;\sum_{j=1}^m|u_j-v_j |, \qquad \qquad \; \forall \boldsymbol{u}, \boldsymbol{v} \in [0,M]^m.
	\label{hpF4}
	\end{eqnarray}
	Thanks to assumption \eqref{hpF2}, solutions $u_i$ are nonnegative, and  \eqref{hpF1} provides us with mass control since the total integral of the solution is bounded with exponential growth in time. 
	\\
	
	We do not consider that the $f_i$'s depend on $(x,t)\in Q_T$, but we could extend these assumptions also to that case. We rather give an example modeling intracellular transport phenomena~\cite{cangiani, dimitrio, serafini} in order to understand the class of systems that we have in mind.
	Molecule trafficking across the nuclear envelope has been studied using reaction-diffusion equations with Kedem-Katchalsky conditions. Small molecules can pass through  nuclear pore complexes (NPCs). The translocation of larger molecules is allowed by a system for active transport across the NPCs. The cargo protein binds to a nucleocytoplasmic transport receptor known as importin, which mediates the transport throught the nuclear envelope. The energy needed is provided by the Ran complex. In order to reproduce this intracellular dynamics, Cangiani and Natalini proposed a model in~\cite{cangiani}. We denote by $\Omega^n$ and $\Omega^c$ respectively the nuclear and the cytoplasmic compartment with $\Gamma^{nc}=\partial\Omega^n$ the interface between them. In each compartment, we can write a system of coupled reaction-diffusion equations of type 
	\begin{equation}\label{ex01}
	\left\{
	\begin{array}{ll}
	\partial_t R_t=d_r\Delta R_t + f_{rt}(R_t, T, T_r),\\[1ex]
	\partial_t R_d=d_r\Delta R_d + f_{rd}(R_t),\\[1ex]
	\partial_t T_r=d_{tr}\Delta T_r + f_{tr}(R_t, T, T_r),\\[1ex]
	\partial_t C=d_c\Delta C + f_c(C,T),\\[1ex]
	\partial_t T=d_t\Delta T + f_{t}(R_t, T, T_r, C),\\[1ex]
	\partial_t T_c=d_{tc}\Delta T_c + f_{tc}(C, T).
	\end{array}
	\right.
	\end{equation}
	The two systems are coupled through Neumann homogeneous boundary conditions and Kedem-Katchalsky transmission conditions.  Reactions have at most quadratic growth and they satisfy hypothesis~\eqref{hpF0}--\eqref{hpF4}. 
	This is only an example of a biological system satisfying our assumptions. Its relevance will bring us to develop numerical results aiming to study biological phenomena fitting with the theory presented in this paper.

	\subsection{Main result}
	
	The aim is to prove global existence when the $f_i$'s are at most quadratic and for a membrane problem as \eqref{model}. As mentioned before, we follow the literature concerning existence results for reaction-diffusion systems by M.~Pierre~\cite{baras, bothe, pierre}, by E.-H. Laamri and M.~Pierre~\cite{lpi} and by E.-H.~Laamri and B.~Perthame~\cite{lp}. 
	A local result in the case of membrane conditions is available but taking into account local Lipschitz reaction terms with $u_0\in H^s$, for~$s>\frac{d}{2}$ (e.g.~\cite{serafini}).
	\\
	
	Our main contribution is the following global existence theorem with initial data of low regularity and reaction terms at most quadratic.
	We first enunciate some definitions and introduce the appropriate test functions space for our problem. We recall that
	$$Q_T =(0,T)\times \Omega, \quad\Sigma_T =(0,T)\times(\Gamma^1\cup\Gamma^2),\quad\Sigma_{T,\Gamma}=(0,T)\times \Gamma.$$
	\begin{defin}
			For $ i=1,...,m,$ we define the space of  test functions 
			\[ \begin{array}{rl}
			\mathcal{D}_i:=\big\{&(\psi^1,\psi^2)\in C^\infty([0,T]\times \overline{\Omega^1})\times C^\infty([0,T]\times \overline{\Omega^2}),	
			\\[5pt]&\psi\geq 0,\; \psi(\cdot,T)=0,\; \psi=0 \mbox{ in } \Sigma_T, \; 	
			\partial_{\boldsymbol{n}^1}\psi^1=\partial_{\boldsymbol{n}^1}\psi^2=k_i(\psi^2-\psi^1) \mbox{ in } [0,T] \times \Gamma \big \},
			\end{array}\]
			where $
			\psi=\left\{\begin{array}{ll}
			\psi^1,& \mbox{ in } \Omega^1,   \\
			\psi^2,& \mbox{ in } \Omega^2. 
			\end{array}\right.
			$
	\end{defin} 
	We investigate the existence of a global weak solution of System \eqref{model} defined by duality as
	\begin{defin}\label{defweaksol}
			We define a weak solution of System \eqref{model} as a function $\boldsymbol{u}=(u_1,...,u_m)$ such that for all $T>0$ and $ i=1,...,m$, $u_i \in L^1(Q_T) $,  $f_i(\boldsymbol{u})\in L^1(Q_T)$ and for $\psi\in \mathcal{D}_i$, it holds
			\begin{equation}\label{weaksol} 
			-\int_{\Omega} \psi(0,x)u_{0,i}+\int_{Q_T} u_i(-\pt \psi -D_i \Delta \psi)= \int_{Q_T} \psi f_i. 
			\end{equation}
	\end{defin}
	
	We consider the space  ${\bf H^1}$ and its dual as in Definitions~\ref{defh1} and~\ref{defh1star}.
		\begin{thm} [Existence and regularity] \label{thmprinc}
			Assume \eqref{hpF0}-\eqref{hpF4} and that $k_1=...=k_m.$
			Then, for all $\boldsymbol{u}_0=(u_{0,1},...,u_{0,m}), $ such that $ \boldsymbol{u}_0\in (L^1(\Omega)^+ \cap ({\bf H^1})^*)^m$, System \eqref{model} has a nonnegative global weak solution in the sense of Definition~\ref{defweaksol} which satisfies for all $T>0$ and $i=1,...,m$,
			\begin{eqnarray}
			&u_i \in L^2(Q_T) \quad \mbox{ and } \quad (1+|u_i|)^\alpha \in L^2\big(0,T; H^1(\Omega)\big), \quad &\forall \alpha \in \left[0,\frac{1}{2}\right), \label{th1thm1}\\
			&u_i \in L^\beta\big(0,T;W^{1,\beta}(\Omega)\big)\quad \mbox{ and } \quad u_i \in L^\beta\big(0,T;L^\beta(\Gamma)\big),\quad &\forall \beta\in \left[1,\frac{d}{d-1}\right).	
			\end{eqnarray}
		\end{thm}
	
	\subsection{Preliminary lemmas and proof organisation}
	\label{prelim}
	
	In order to prove this result, we follow four main steps according to Pierre's method.
	\\
	\\
	{\em First step. Regularization process.}
	We build a regularized problem with a nonnegative classical global solution $\boldsymbol{u}^n$.
	\\
	\\	
	{\em Second step. An $L^2$ lemma.}
	We  extend the Laamri-Perthame \cite{lp} a priori $L^2$ estimate of the solution given an $L^1$ initial data to the case of membrane conditions (see Subsection~\ref{step2}). In particular, we gain
	
	\begin{lemma} [Key estimate with $L^1$ data and membrane conditions]
		Consider smooth functions $z_i:[0,+\infty)\times \Omega \rightarrow \R^+$, $f_i:[0,+\infty)^m \rightarrow \R$,  for all	$i=1,...,m$, with $f_i$ satisfying the assumption~\eqref{hpF1}. Assume $z_{0,i}\in L^1(\Omega) \cap ({\bf H^1})^*$ and that the equation holds with~$k_i = k$
		\begin{equation}
		\left\{\begin{array}{ll}
		\partial_t z_i-D_i\Delta z_i=f_i(z_1,...,z_m),& \mbox{in}\; Q_T,\\[1ex]
		z_i=0,& \mbox{in}\; \Sigma_T,\\[1ex]
		\partial_{\boldsymbol{n}^{1}} z_i^{1}=\partial_{\boldsymbol{n}^{1}} z_i^2=k_i(z_i^2-z_i^{1}), & \mbox{in}\; \Sigma_{T,				\Gamma},\\[1ex]
		z(0,x)=z_{0,i}(x)\geq 0, & \mbox{in}\; \Omega.
		\end{array}
		\right.  
		\end{equation}
		Then, for some constant  $C_3$ depending on $\|\boldsymbol{z}_{0}\|_{{({\bf H^1})^*}}$, the inequality holds
		$$
		\sum_{i=1}^m\;\int_{Q_T} |z_i|^2\leq C_3.
		$$
	\end{lemma}		
	
	From this lemma we derive an $L^1$ bound for the reaction term $\boldsymbol{f}^n(\boldsymbol{u}^n)$ of the regularized system thanks to \eqref{hpF0}.
	The proof uses the solution of  an elliptic problem $-\Delta w=f$ with membrane conditions which has a unique solution thanks to the Lax-Milgram theorem (see \cite{evans}, p.$297$) and we recall its statement in our context. 
	\\
	
	We assume $H$ a real Hilbert space with norm $\|\cdot\|$ and inner product $(\cdot,\cdot)$. Let $\langle\cdot,\cdot\rangle$ denote the pairing of $H$ with its dual space. 
	
	\begin{thm} [Lax-Milgram theorem] \label{lm}
		Given $B: H \times H \rightarrow \R$, a bilinear mapping for which there exist constants $\gamma, \delta >0$ such that for all $w,z\in H$,
		$$
		|B[w,z]|\leq \gamma \|w\| \;\|z\| \quad \mbox{ (continuity) }, \qquad \qquad  |B[w,w]|\geq \delta \|w\|^2 \quad \mbox{ (coercivity)}.
		$$
		Finally, let $f:H\rightarrow\R$ be a bounded linear functional on $H$. Then there exists a unique $w\in H$ such that 
		$$
		B[w,z]=\langle f,z \rangle, \quad \forall z\in H.
		$$
	\end{thm}
	We can apply the Lax-Milgram theorem for membrane problems (see \cite{serafini}).	In order to justify this, we introduce some definitions. The first ones concern the space $H= {\bf H^1}$ under consideration,  the second is the bilinear form. 
	
	\begin{defin}\label{defh1}
		We define ${\bf H^1} = H^1_{0,\Gamma}(\Omega^1)\times H^1_{0,\Gamma}(\Omega^2)$ as the Hilbert space of functions $H^1(\Omega^1)\times H^1(\Omega^2)$ satisfying Dirichlet homogeneous conditions on $\Gamma^\lambda$, $\lambda=1,2$.	
		We endow it with the norm 
		$$
		\|w\|_{\bf H^1}= \left(\|w^1\|^2_{H^1(\Omega^1)} +\|w^2\|^2_{H^1(\Omega^2)}\right)^\frac{1}{2}.
		$$
		We let $(\cdot,\cdot)$ be the inner product in ${\bf H^1}$ and $\langle\cdot,\cdot\rangle$ denote the pairing of ${\bf H^1}$ with its dual space.
	\end{defin}	
	
	\begin{defin}\label{defh1star}
		We introduce the dual space of ${\bf H^1}$ as $({\bf H^1})^* =  \big( H^1_{0,\Gamma}(\Omega^1)\times H^1_{0,\Gamma}	(\Omega^2)\big)^* =  H^1_{0,\Gamma}(\Omega^1)^*\times H^1_{0,\Gamma}(\Omega^2)^* $.
	\end{defin} 
	
	Now, we define a proper bilinear form associated to the Laplacian operator considering Dirichlet conditions on $\Gamma^\lambda$, $\lambda=1,2$ and membrane conditions on $\Gamma$.
	
	\begin{defin}
		We consider the continuous, coercive bilinear form $B: {\bf H^1} \times {\bf H^1} \rightarrow \R$, such that 
		$$
		B[w,z]= \int_\Omega \nabla w\nabla z + \int_\Gamma k_i (w^2-w^1)(z^2-z^1), \quad \mbox{ for } w,z \in {\bf H^1}.
		$$
	\end{defin} 
	
	We can readily check continuity and coercivity.
	\\
	{\it $B$ is continuous:} thanks to the Cauchy-Schwarz inequality and the continuity of the trace, we can write
	\begin{align*}
	|B[w,z]|&\leq \sum_{1\leq\lambda\leq2}\nl{\nabla w^\lambda}{\Omega^\lambda} \nl{\nabla z^\lambda}{\Omega^\lambda} +C k_i \nl{\,[w]\,}{\Gamma} \nl{\,[z]\,}{\Gamma}\\ &\leq \sum_{1\leq\lambda,\sigma\leq2}\left( \nh{w^\lambda}{\Omega^\lambda}\nh{z^\lambda}{\Omega^\lambda}+C k_i \nh{w^\lambda}{\Omega^\lambda} \nh{z^\sigma}{\Omega^\sigma}\right)\\
	&\leq C\|w\|_{\bf H^1}\|z\|_{\bf H^1},
	\end{align*}	
	\\	
	{\it $B$ is coercive: }	indeed, we can estimate
	$$
	B[w,w] = \int_\Omega |\nabla w|^2 + \int_\Gamma k_i |w^2-w^1|^2 \geq C \|w\|_{\bf H^1}^2,
	$$
	since, thanks to the Dirichlet conditions on $\Gamma^\lambda$, $\lambda=1,3$, and to Theorem \ref{poin}, we have
	$$
	\|w^\lambda\|_{H^1(\Omega^\lambda)} \leq C \nl{\nabla w^\lambda}{\Omega^\lambda}, \quad \mbox{ for } \lambda=1,2.
	$$
	Therefore, using the Lax-Milgram theorem, taking an $L^2$ right-hand side, the elliptic membrane problem has a unique solution $w\in {\bf H^1}$ and, thanks to the Riesz–Fréchet representation theorem (\cite{brezis}, p.$135$) and to the equivalence of the norm $B[w,w]^\frac{1}{2}$ and the original one in ${\bf H^1}$, we have
	\begin{equation}\label{dualnorm}
	\|f\|_{({\bf H^1})^*}=B[w,w]^\frac{1}{2}.
	\end{equation}
	Moreover,   throughout the paper,  we are also allowed  to integrate by parts functions in the Hilbert space ${\bf H^1}$,  considering also the membrane. 
	\\
	\\	 
	{\em Third step. Existence of a global weak supersolution.} 	
	We prove a first theorem which asserts the convergence in $L^1(Q_T)$ of $\boldsymbol{u}^n$ to a supersolution of System~\eqref{model}. 
	Another central result is the following compactness lemma which explains the regularity stated in Theorem~\ref{thmprinc} (see Appendix~\ref{appendix_reg}~and~\ref{appendix_compact}),
	
	\begin{lemma} [A priori bounds] \label{compreglemma1}
		We consider $w$ solution of the problem in dimension $d\geq 2$
		\begin{equation}\label{pb21}
		\left\{\begin{array}{ll}
		\partial_t w-D\Delta w=f,& \mbox{in}\; Q_T,\\[1ex]
		w=0,& \mbox{in}\; \Sigma_T,\\[1ex]
		\partial_{\boldsymbol{n}^{1}} w^{1}=\partial_{\boldsymbol{n}^{1}} w^2=k(w^2-w^{1}), & \mbox{in}\; \Sigma_{T,\Gamma},
		\\[1ex]
		w(0,x)=w_0(x)\geq 0, & \mbox{in}\; \Omega,
		\end{array}
		\right.  
		\end{equation}
		with $f\in L^1(Q_T)$ and $w_0\in L^1(\Omega)$. Then, 
		\begin{itemize}
			\item $w \in L^{\beta}\big(0,T;W^{1,\beta}(\Omega)\big), \; \forall \beta \in \left[1,\frac{d}{d-1}\right)$ and $(1+|w|)^\alpha \in L^2\big(0,T;H^1(\Omega)\big)\; \mbox{ for } \alpha \in \left[0,\frac{1}{2}\right)$.
			\item The mapping $(w_0,f)\longmapsto w$ is compact from $L^1(\Omega)\times L^1(Q_T)$ into $L^1\big(0,T;L^{\gamma_1} (\Omega)\big)$, for all  $\gamma_1<\frac{d}{d-2}$ and $L^{\gamma_2}(Q_T)$ for all   ${\gamma_2}  < \frac{2+d}{d}$.
			\item The trace mapping $(w_0,f)\longmapsto Tr_\Gamma(w)\in L^\beta\big(0,T;L^\beta(\Gamma)\big),\; \beta \in \left[1,\frac{d}{d-1}\right)$ is also compact.
		\end{itemize}
	\end{lemma}
	
	\noindent {\em Fourth step. Existence of a global weak solution.}
	We conclude with a second theorem asserting the convergence in $L^1(Q_T)$ of $u_i^n$, $i=1,...,m$ to a solution of System \eqref{model}.

	\section{Proof of the existence result}
	\label{proof}
	
	We are now ready to prove Theorem \ref{thmprinc} according to the previous steps.
	
	\subsection{Regularized problem}
	\label{regpb}
	
	First of all, we approximate the initial data and the reaction term as 
	\begin{equation}\label{u0n-fn}
	u_{0,i}^n:=\varphi_{\delta_n} \ast \inf\{u_{0,i},n\}
	\quad\mbox{ and }\quad\displaystyle f_i^n({\bf u}^n):=\frac{f_i({\bf u}^n)}{1\;+\;\frac{1}{n}\;\sum_{1\leq j\leq m} |f_j({\bf u}^n)|} .	
	\end{equation} 
	For the initial data, we consider a regularized version thanks to a convolution with a mollifier sequence $\varphi_{\delta_n} $ which is only used to assert existence in the framework of~\cite{serafini}.
	We readily check that $\boldsymbol{f}^n$ satisfies assumptions \eqref{hpF0}-\eqref{hpF4}. In particular, for \eqref{hpF4},  there is a  $C_M$ such that
	\begin{equation}
	|f_i^n(\boldsymbol{u})-f_i^n(\boldsymbol{v})|\leq C_M\;\sum_{i=1}^m|u_i-v_i|, \quad\forall \boldsymbol{u}, \boldsymbol{v}\in [0,M]^m.
	\end{equation}
	Moreover, we have 
	\begin{equation}
	|f_i^n|\leq n \quad \mbox{ and } \quad \epsilon^n_M:= \sup_{\boldsymbol{u}\in [0,M]^m, i=1,2} |f_i^n(\boldsymbol{u})-f_i(\boldsymbol{u})|\,\leq\, \frac{C(M)m}{n}.
	\end{equation}
	
	We consider an approximation of System~\eqref{model}, for all $i=1,...,m$,
	\begin{equation}\label{modelapprox}
	\left\{
	\begin{array}{ll}
	\partial_t u_i^n-D_i\Delta u_i^n=f_i^n(u^n_1,..., u^n_m),& \mbox{in}\; Q_T,\\[1ex]
	u_i^n=0,& \mbox{in}\; \Sigma_T,\\[1ex]
	\partial_{\boldsymbol{n}^1} u^{n,1}_i=\partial_{\boldsymbol{n}^1} u^{n,2}_i=k_i(u_i^{n,2}-u_i^{n,1}), & \mbox{in}\; \Sigma_{T,\Gamma},\\[1ex]
	u_i^n(0,x)=u^n_{0,i}(x), & \mbox{in}\; \Omega.
	\end{array}
	\right.
	\end{equation}	 
	Since $\boldsymbol{f}^n$ is uniformly bounded for fixed $n$,  from \cite{serafini} we know that there exists a global classical solution $\boldsymbol{u}^n=(u^{n}_1,..., u^{n}_m)$ to~\eqref{modelapprox}.	
	
	\subsection{The $L^2$ lemma with membrane conditions}
	\label{step2}
	
	The second step of the proof is to apply to $u_i^n$, $i=1,...,m$ the following Laamri-Perthame~\cite{lp} version of Pierre's  lemma, adding our membrane conditions.
	
	\begin{lemma}[Key estimate with $L^1$ data and membrane conditions] \label{lemmalp}
		Consider smooth functions $z_i:[0,+\infty)\times \Omega \rightarrow \R^+$, $f_i:[0,+\infty)^m \rightarrow \R$, for all $i=1,...,m$,  with $f_i$ satisfying the assumption \eqref{hpF1}. Assume $z_{0,i}\in L^1(\Omega) \cap ({\bf H^1})^*$ and that the differential equation holds with $k_i = k$
		\begin{equation}\label{senzahat}
		\left\{\begin{array}{ll}
		\partial_t z_i-D_i\Delta z_i=f_i(z_1,...,z_m),& \mbox{in}\; Q_T,\\[1ex]
		z_i=0,& \mbox{in}\; \Sigma_T,\\[1ex]
		\partial_{\boldsymbol{n}^{1}} z_i^{1}=\partial_{\boldsymbol{n}^{1}} z_i^2=k_i(z_i^2-z_i^{1}), & \mbox{in}\; \Sigma_{T,\Gamma},\\[1ex]
		z(0,x)=z_{0,i}(x)\geq 0, & \mbox{in}\; \Omega.
		\end{array}
		\right.  
		\end{equation}
		Then, for some constant $C_3$ depending on $\|\boldsymbol{z}_{0}\|_{{({\bf H^1})^*}}$, the inequality holds
		$$
		\sum_{i=1}^m\;\int_{Q_T} |z_i|^2\leq C_3.
		$$
	\end{lemma}
	It is an open problem to extend it to the case where the constants $k_i$ are different and it is also noticeable that the other proofs (time integration or duality) also apply only with the condition $k_i=k$. 	
	\begin{proof}
		We consider $\widehat{u}_i=e^{-Ct}z_i$ for $i=1,...,m$,  where C is the same constant than in \eqref{hpF1}. Substituting in the equation for $z_i$, we obtain that for all $i=1,...,m$,
		$$
		\partial_t \widehat{u}_i-D_i\Delta \widehat{u}_i=e^{-Ct}[f_i(z_1,...,z_m)-C z_i],
		$$
		with the same boundary  and initial conditions as in \eqref{senzahat} but for $\h{u}_i.$ Adding up and defining 
		$$
		\h{U}=\sum\limits_{i=1}^{m} \widehat{u_i}, \qquad \h{V}=\sum\limits_{i=1}^{m} D_i\widehat{u_i},
		$$
		we obtain
		\begin{equation}\label{hat}
		\partial_t \h{U}-\Delta \h{V} =e^{-Ct}[\sum_{i=1}^{m}f_i(z_1,...,z_m)-C\sum_{i=1}^{m}z_i]\leq Ce^{-Ct}\leq C, \quad \mbox{ in } Q_T    
		\end{equation}
		with conditions
		$$ \left\{\begin{array}{ll}
		\h{U}= 0,& \mbox{in}\; \Sigma_T,\\[1ex]
		\partial_{\boldsymbol{n}^{1}} \h{U}^{1}=\partial_{\boldsymbol{n}^{1}} \h{U}^2= k(\h{U}^2-\h{U}^{1}), & \mbox{in}\; \Sigma_{T,\Gamma},\\[1ex]
		\h{U}(0,x)=\h{U}_{0}(x)\geq 0, & \mbox{in}\; \Omega.
		\end{array}
		\right.  $$
		Thanks to the Lax-Milgram theorem \ref{lm} (see also \cite{BCF}),  we may define the solution of 
			$$ \left\{\begin{array}{ll}
			-\Delta \h{W}=\h{U}, &\mbox{in } Q_T\\[1ex]
			\h{W}=0,& \mbox{in } \Sigma_T,\\[1ex]
			\partial_{\boldsymbol{n}^{1}} \h{W}^{1}=\partial_{\boldsymbol{n}^{1}} \h{W}^2= k(\h{W}^2-\h{W}^{1}), & \mbox{in } \Sigma_{T,\Gamma}.\\[1ex]
			\end{array}
			\right.  $$
			So, at this point, with $\underline{G}=\partial_t \h{W} +\h{V}$,  we can write \eqref{hat} as an elliptic inequality 
			$$
			\left\{\begin{array}{ll}
			-\Delta \underline{G}\leq C, &\qquad  \mbox{in } Q_T,
			\\[1ex]
			\underline{G}= 0,&\qquad  \mbox{in }\Sigma_T,
			\\[1ex]
			\partial_{\boldsymbol{n}^{1}} \underline{G}^{1}=\partial_{\boldsymbol{n}^{1}} \underline{G}^2 &\!= k(\partial_t \h{W}^2-\partial_t\h{W}^{1})+k \sum\limits_{i=1}^m D_i(\h{u}_i^2-\h{u}_i^{1}) \\[1ex]
			&\!=k[\partial_t \h{W}] + k[\h{V}]=k(\underline{G}^{2}-\underline{G}^{1}), \qquad  \mbox{in }\; \Sigma_{T,\Gamma}. 
			\end{array}
			\right.    
			$$
		Lax-Milgram theorem \ref{lm} allows us to state the existence of a function $G\in \mathbf{H}^1$ satisfying the system
		\[\left\{\begin{array}{ll}
			-\Delta {G}= C, &  \mbox{in } Q_T,
			\\[1ex]
			{G}= 0,&  \mbox{in }\Sigma_T,
			\\[1ex]
			\partial_{\boldsymbol{n}^{1}} {G}^{1}=\partial_{\boldsymbol{n}^{1}} {G}^2 =k({G}^{2}-{G}^{1}), &  \mbox{in }\; \Sigma_{T,\Gamma}. 
		\end{array}
		\right.    \]
		By comparison theorem~\cite{serafini}, we conclude that $\underline{G}\leq G$, in $\overline{Q}_T$. 
		So, multiplying $\underline{G}$ by $\h{U}$ and integrating over space, we compute, since $\h{U}=-\Delta \h{W}$ and $\underline{G}\leq G$,
		\[\int_\Omega \h{U}\underline{G}=-\int_\Omega \Delta\h{W}\pt \h{W} \;+\;\int_\Omega \h{U}\h{V} \leq \int_\Omega \h{U} G \leq \int_\Omega \frac{D\h{U}^2}{2} + \frac{G^2}{2D}\leq \frac{1}{2} \int_\Omega \h{U}\h{V} +C_1,\]
		thanks to Young's inequality applied to $\sqrt{D} \h{U}$ and $\frac{G}{\sqrt{D}}$ with $D=\min_{i=1,...,m} D_i>0$, see~\cite{brezis}, the fact that by definition $D\h{U}\leq \h{V}$, and the $L^2$-bound of~$G$. Then, reorganising the terms on the right and left hand-side, we derive
		\[-\int_\Omega \Delta\h{W}\pt \h{W} \;+\;\frac{1}{2}\int_\Omega \h{U}\h{V} \leq C_1.\]
		Following Subsection~\ref{prelim} and the definition of the Hilbert space  ${\bf H^1}$ (see Definition~\ref{defh1}), we can integrate by parts obtaining
		$$
		\frac{1}{2} \frac{d}{dt} \int_{\Omega} |\nabla \h{W}|^2 \;+\;\frac{1}{2}\int_\Omega \h{U}\h{V}\;\leq\; \int_{\p\Omega}\p_n\h{W} \pt\h{W}+C_1. 
		$$ 
		Next, we remark that
		$$\begin{array}{rl}
		\displaystyle \int_0^T\! \! \int_{\p\Omega}\p_n\h{W} \pt\h{W} &=\displaystyle \int_0^T \!\!  \int_{\Gamma}\p_{\boldsymbol{n}^1}\h{W}^1 (\pt\h{W}^1-\pt\h{W}^2)
		\\[2ex]
		&\displaystyle  =-\int_0^T \! \!\int_{\Gamma}k(\h{W}^2-\h{W}^1) \pt(\h{W}^2-\h{W}^1)=-\frac{k}{2} \int_0^T  \! \frac{d}{dt}\int_\Gamma (\h{W}^2-\h{W}^1)^2.
		\end{array}$$
		Therefore, integrating in time and using the relation \eqref{dualnorm}, we arrive to
		\begin{equation}\label{dimlemma}
		\frac{1}{2} \|\h{U}(T)\|^2_{({\bf H^1})^*} \;+\; \frac{1}{2}\int_{Q_T} \h{U}\h{V}\;\leq\;\frac{1}{2} \|\h{U}_0\|^2_{({\bf H^1})^*} + C_1 .
		\end{equation}
		Finally, thanks to Equation~\eqref{dimlemma}, we can assert that
		$$
		\sum\limits_{i=1}^{m} D_i\int_{Q_T} \widehat{u_i}^2\leq C_2.
		$$
		This concludes the proof of Lemma~\ref{lemmalp} since  $z_i^2= e^{2Ct}\widehat{u_i}^2$.
		
	\end{proof}
	
	\subsection{Existence of a global weak supersolution}
	\label{step3}
	
	At this point we can complete the existence result of Theorem~\ref{thmprinc}, since, thanks to Lemma~\ref{lemmalp} and to assumption  \eqref{hpF0}, we know that  the reaction term $\boldsymbol{f}^n$ is bounded in $L^1$. With this in hands, we can assert the existence of a supersolution of System \eqref{model}.
	
		\begin{thm} [Existence of a supersolution] \label{thmsupersol}
			Let $\boldsymbol{u}^n=(u^n_1,...,u^n_m)$ be a nonnegative solution of the approximate System~\eqref{modelapprox}. Consider $k_1=...=k_m$. 
			As defined in \eqref{u0n-fn},  $f_i^n(\boldsymbol{u}^n)$ is bounded in $L^1(Q_T)$, for  $i=1,...,m$ and ${\bf u}^n_0 \rightarrow {\bf u}_0$ in $L^1(\Omega)$.
			Then, up to a sub-sequence, $\boldsymbol{u}^n$ converges in $L^1(Q_T)$ and a.e. to a supersolution $\boldsymbol{u}$ of System \eqref{model} which means that for $i=1,...,m$, and $\beta\in \left[1,\frac{d}{d-1}\right)$,
			$$
			f_i(\boldsymbol{u})\in L^1(Q_T), \quad  u_i \in L^\beta\big(0,T; W^{1,\beta}(\Omega)\big), \quad Tr_\Gamma ( u_i) \in L^\beta\big(0,T; L^\beta(\Gamma)\big),
			$$
			\begin{equation}\label{supersol}
			-\int_{\Omega} \psi(0,x)u_{0,i}+\int_{Q_T} (-\psi_t u_i+D_i\nabla\psi\nabla u_i)+\int_0^T\int_{\Gamma}D_i k_i[u_i][\psi]\geq \int_{Q_T} \psi f_i,
			\end{equation}
			for all  $\psi\in \mathcal{D}_i$, $\psi \geq 0$. 
	\end{thm}
	
	\begin{proof} 
		We divide the proof in several steps which are adaptations from Pierre's method. 
		\\
		\\
		\noindent {\bf Compactness of $\boldsymbol{u}^n$ and $Tr_\Gamma(\boldsymbol{u}^n)$.} Combining  Lemma~\ref{lemmalp} and assumption \eqref{hpF0}, we notice that  $f_i^n(\boldsymbol{u}^n)$ is bounded in $L^1(Q_T)$ for $i=1,...,m$.

		Next, we  apply  the compactness Lemma~\ref{compreglemma1} (see also Lemma~\ref{compreglemma} and its proof in Appendix~\ref{appendix_reg},~\ref{appendix_compact}) to the solution $\textbf{u}^n$ of the approximate System~\eqref{modelapprox}.  Accordingly, after extraction, the following convergences, hold 
		\begin{equation}\label{convsol}
		\left\{\begin{array}{l l l}
		\boldsymbol{u}^n \rightarrow \boldsymbol{u}, & \mbox{ in } L^1\big(0,T;L^{\gamma_1}(\Omega)\big)^m, & \forall \gamma_1 \in \Big[1,\frac{d}{d-2}\Big),  
		\\[1ex]
		\boldsymbol{u}^n \rightarrow \boldsymbol{u}, &  \mbox{ a.e. in} \; Q_T,  
		\\[1ex]
		\nabla \boldsymbol{u}^n \rightharpoonup \nabla \boldsymbol{u}, & \mbox{ in } [ L^{\beta}(Q_T)^d]^m, & \forall \beta \in \left[1,\frac{d}{d-1}\right),
		\\[1ex]
		Tr_\Gamma(\boldsymbol{u}^n) \rightarrow Tr_\Gamma(\boldsymbol{u}), & \mbox{ in } L^1\big(0,T;L^{\beta}(\Gamma)\big)^m, & \forall \beta \in \left[1,\frac{d}{d-1}\right).
		\end{array}
		\right.
		\end{equation}
		
		\noindent {\bf Pointwise convergence of the  $f_i^n$'s.} 
		Since $u_i^n$ satisfies \eqref{weaksol} for all $i=1,...,m$, i.e.
		\begin{equation}\label{nweaksol}
		-\int_{\Omega} \psi(0,x)u^n_{0,i}+\int_{Q_T} (-\psi_t u_i^n+D_i\nabla\psi\nabla u_i^n)+\int_0^T\int_{\Gamma}D_i k_i[u_i^n][\psi]= \int_{Q_T} \psi f_i^n, 
		\end{equation}
		and our goal is to pass to the limit as $n\rightarrow +\infty$, we need to study the convergence of $f_i^n$.
		\\
		Thanks to the choice of $\boldsymbol{f}^n$: a.e. convergence of $\epsilon_M^n$ to zero and the continuity with respect to its argument, we infer
		$$f_i^n(\boldsymbol{u}^n)\rightarrow f_i(\boldsymbol{u}) \mbox{ a.e. in } Q_T.$$
		By Fatou's lemma, we know that
		$$\int_{Q_T} |\boldsymbol{f}(\boldsymbol{u})| \leq \liminf_{n\rightarrow +\infty} \int_{Q_T} |\boldsymbol{f}^n(\boldsymbol{u}^n)|$$
		and, in particular, it holds 
		$$\boldsymbol{f}(\boldsymbol{u})\in L^1(Q_T)^m.$$
		So far we did not prove $L^1$-convergence of $f_i^n(\boldsymbol{u}^n)$, therefore we cannot pass to the limit in the Equation~\eqref{nweaksol} obtaining a weak solution of System \eqref{model}. However we can find an inequality in the formulation of the weak solution of System \eqref{model}, thus obtaining a supersolution. We arrive at this applying a truncation method. \\[1ex]
		
		\noindent{\bf Truncation method.}
		The idea is that, with an appropriate truncation, we succeed in obtaining a reaction-diffusion inequality in which the reaction terms
		are under control  as $n\rightarrow +\infty$ with a fixed  truncation level. In this way, we are able to pass to the limit in the truncated weak supersolution formula, as $n\rightarrow +\infty$. At this point, bringing the truncation level to infinity, we gain the supersolution property in Theorem~\ref{thmsupersol}.
		
		In order to build the truncation $T_b$ at level $b$, since we will have to differentiate twice $T_b$, we replace $T_b$ by a $C^2$-regularized version (otherwise $T_b''$ would be a Dirac mass), still denoted by $T_b$, so that on $[0,+\infty)$ we have
		$$
		0\leq T'_b\leq 1, \;\quad -1\leq T''_b\leq 0, \;\quad T_b(\sigma)=\sigma \quad\forall \sigma\in [0,b], \;\quad T'_b(\sigma)=0 \quad \forall \sigma \in (b,+\infty).
		$$ 
		We fix $\eta\in (0,1)$ and we denote for all $i=1,...,m$,
		$$U^n_i=\sum_{j\ne i} u_j^n, \quad W^n_i=u^n_i+\eta U^n_i.$$
		The idea is to consider the limit for $n\rightarrow +\infty$, then $\eta \rightarrow 0$ and, finally, $b\rightarrow +\infty$.
		
		\noindent The main point is to use the inequality satisfied by $v^n :=  T_b(W^n_i)$, taking into account the previous properties of $T'_b$ and $T''_b,$
		\begin{align*}
		-\Delta v^n = -\Delta T_b(u_i^n+\eta U^n_i)&=-T''_b (u_i^n+\eta U^n_i)|\nabla u_i^n+\eta\nabla U^n_i|^2\,-\,T'_b(u_i^n+\eta U^n_i)[\Delta u_i^n+\eta\Delta U^n_i]\\[1ex]
		&\geq -T'_b(u_i^n+\eta U^n_i)[\Delta u_i^n+\eta\Delta U^n_i].
		\end{align*}
		\noindent This implies
		$$v^n_t-D_i\Delta v^n\geq T'_b(u_i^n+\eta U^n_i)[f_i^n+\eta \sum_{j\ne i} f^n_j]+\eta T'_b(u_i^n+\eta U^n_i)\sum_{j\ne i}(D_j-D_i)\Delta u^n_j =:R^n_i+\eta S^n_i,$$
		where 
		\begin{equation} \label{def:RS}
		R^n_i=T'_b(u_i^n+\eta U^n_i)[f_i^n+\eta \sum_{j\ne i} f^n_j], \qquad S^n_i=T'_b(u_i^n+\eta U^n_i)\sum_{j\ne i}(D_j-D_i)\Delta u^n_j.
		\end{equation}
		So the truncation $T_b(W^n_i)$ solves the problem
		{\small\begin{equation}
			\left\{
			\begin{array}{ll}
			v^n_t-D_i\Delta v^n\geq R^n_i+\eta S^n_i, \\[1ex]
			\begin{array}{ll}
			v^n_{|_{\Gamma^{\lambda}}}&=0,\quad \lambda=1,2,\\[1ex]
			\partial_{\boldsymbol{n}^1} v^{n,1}_{|_\Gamma}&=T'_b(u^{n,1}_i+\eta U^{n,1}_i)[\partial_{\boldsymbol{n}^1} u^{n,1}_i+\eta \partial_{\boldsymbol{n}^1} U^{n,1}_i]\\[1ex]
			&=T'_b(u^{n,1}_i+\eta U^{n,1}_i)[k_i(u^{n,2}_i-u^{n,1}_i) +\eta \sum_{j\ne i} k_j(u^{n,2}_j-u^{n,1}_j)]=:T'_{b,n,1} V^n_{i},\\[1ex]
			\partial_{\boldsymbol{n}^1} v^{n,2}_{|_\Gamma}&=T'_b(u^{n,2}_i+\eta U^{n,2}_i)[\partial_{\boldsymbol{n}^1} u^{n,2}_i+\eta \partial_{\boldsymbol{n}^1} U^{n,2}_i]\\[1ex]
			&=T'_b(u^{n,2}_i+\eta U^{n,2}_i)[k_i(u^{n,2}_i-u^{n,1}_i) +\eta \sum_{j\ne i} k_j(u^{n,2}_j-u^{n,1}_j)]=:T'_{b,n,2} V^n_{i},\\[1ex]
			v^n(0,x)&=T_b\big(u_i^n(0,x)+\eta U^n_i(0,x)\big).
			\end{array}
			\end{array}
			\right.
			\end{equation}}
		Consequently, we may write for all $i=1,...,m$, for all $\psi\in \mathcal{D}_i,$
		$$-\int_\Omega \psi(0) v^n(0)-\int_{Q_T} \psi_t v^n -\int_0^T\int_\Gamma D_i(\psi^1\partial_{\boldsymbol{n}^1} v^{n,1}-\psi^2\partial_{\boldsymbol{n}^1}v^{n,2})+D_i\int_{Q_T}\nabla v^n \nabla \psi $$$$\geq \int_{Q_T}(R^n_i+\eta S^n_i)\psi , $$
		$$-\int_\Omega \psi(0) v^n(0)+\int_{Q_T} (-\psi_t v^n+D_i\nabla v^n \nabla \psi)-\int_0^T\int_\Gamma D_i V^n_{i}(\psi^1 T'_{b,n,1}-\psi^2 T'_{b,n,2})$$
		\begin{equation}\label{soldebv}
		\geq \int_{Q_T}(R^n_i+\eta S^n_i)\psi. 
		\end{equation}
		So, as we said, the truncated function is a supersolution but with reaction terms (see the following) converging in $L^1$ or bounded independently from $n.$\\[1ex]
		$\bullet${ \it Limit for $n\rightarrow +\infty$ with $b,\eta$ fixed.} \\[1ex]
		\indent Since $\boldsymbol{u}^n$ was a convergent solution (see \eqref{convsol}) and $T_b(W^n_i)$ represents the truncation at level $b$ with $b$ fixed, by the dominated convergence theorem,
		$$v^n=T_b(W^n_i) \stackrel{n\rightarrow +\infty}\longrightarrow T_b(W_i)=T_b(u_i+\eta U_i)\;\;\mbox{ in } L^1(Q_T) \mbox{ and a.e.}.$$
		Since $T'_b(\sigma)=0$ for $\sigma>b$, by definition, it holds $R^n_i=0$ on the set $u_i^n+\eta U^n_i>b$. But on $u_i^n+\eta U^n_i \leq b, $  for $s=1,...,m,\;$ $u^n_s$ are uniformly bounded. In fact,
		\begin{equation}\label{unibound}
		u_i^n\leq b \mbox{ and } u^n_j\leq \frac{b}{\eta}, \;\; \forall j\ne i.
		\end{equation}
		By the dominated convergence theorem, using \eqref{hpF1}, we find
		$$R^n_i \stackrel{n\rightarrow \infty}\longrightarrow R_i:=T'_b(u_i+\eta U_i)[f_i+\eta \sum_{j\ne i}f_j]\;\; \mbox{ in } L^1(Q_T).$$
		On the other hand, we remark that
		$$\nabla v^n=\nabla T_b(W^n_i)=T'_b(u_i^n+\eta U^n_i)[\nabla u_i^n+\eta \nabla U^n_i]\rightharpoonup \nabla v=T'_b(u_i+\eta U_j)[\nabla u_i+\eta\nabla U_j]\;\; \mbox{ in } L^1(Q_T)$$
		and we have also convergence of the traces on $\Gamma$ and $\Gamma^\lambda, \; \lambda=1,2$.
		Therefore, to pass to the limit as $n\rightarrow +\infty$ in \eqref{soldebv}, we only need to control $\int_{Q_T} \psi S^n_i.$
		We have (\hyperlink{proofcontrolS}{see the proof later on})
		
		\begin{lemma}(\cite{pierre})\label{controlS}
			There exists C depending on $b,\psi$ and the data, but not on $n$, $\eta \in (0,1)$ such that
			$$
			\left|\int_{Q_T}\psi\; S^n_i\right|\leq C\eta^{-\frac{1}{2}}.
			$$
		\end{lemma}
		
		\noindent So we can pass to the limit as $n\rightarrow +\infty$ in \eqref{soldebv} with $b,\eta$ fixed and we obtain
		$$-\int_\Omega \psi(0) v(0)+\int_{Q_T} (-\psi_t v+D_i\nabla v \nabla \psi)-\int_0^T\int_\Gamma D_i V_{i}(\psi^1 T'_b(W^{1}_i)-\psi^2 T'_b(W^{2}_i))$$
		$$\geq \int_{Q_T}R_i\psi+\eta \int_{Q_T} S^n_i\psi \geq \int_{Q_T}R_i\psi - C\eta^\frac{1}{2},$$
		with $V_{i}= [k_i(u^2_i-u^1_i)+\eta\sum_{j\ne i} k_j(u^2_j-u^1_j)]$.\\[1ex]
		\noindent $\bullet${ \it Limit for $\eta\rightarrow 0$ with $b$ fixed.}  Then, $W_i \rightarrow u_i, \;$ $V_{i}\rightarrow b_i(u^2_i-u^1_i)$ and $R_i\rightarrow T'_b(u_i) f_i$.\\[1ex]
		\noindent $\bullet${ \it Limit for $b\rightarrow +\infty$.} Then, the truncation is converging to the function itself and its derivative to $1$ and so we obtain the statement \eqref{supersol}:
		$$-\int_{\Omega} \psi(0,x)u_{0,i}+\int_{Q_T} (-\psi_t u_i+D_i\nabla\psi\nabla u_i)+\int_0^T\int_{\Gamma}D_i k_i[u_i][\psi]\geq \int_{Q_T} \psi f_i.$$
	\end{proof}
	We now turn to the \hypertarget{proofcontrolS}{{\bf proof of Lemma~\ref{controlS}}}.
	\begin{proof}
		Remembering \eqref{def:RS}, 
		in order to prove Lemma~\ref{controlS}, we need that
		$$\left|\int_{Q_T} \psi T'_b (W^n_i)\sum_{j\ne i}(D_j-D_i)\Delta u^n_j\right|\leq C\eta^{-\frac{1}{2}}.$$
		Consequently, we have to study the following integral
		$$
		\begin{array}{ll}
		\int_{Q_T} \psi T'_b(W^n_i)\Delta u^n_j &=
		\int_{Q_T} \mbox{div}(\psi T'_b(W^n_i)\nabla u^n_j)-\int_{Q_T} \mbox{div}(\psi T'_b(W^n_i))\nabla u^n_j\\[1ex]
		&= \int_0^T\int_{\Gamma} (\psi^1 T'_{b}(W^{n,1}_i) \partial_{\boldsymbol{n}^1} u^{n,1}_j+\psi^2 T'_b(W^{n,2}_i)\partial_{\boldsymbol{n}^2} u^{n,2}_j)\\[1ex]
		&- \int_{Q_T} [T'_b(W^n_i) \nabla \psi+\psi T''_b(W^n_i) \nabla W^n_i]\nabla u^n_j.
		\end{array}
		$$
		We remark that
		$$\left|\int_0^T\int_{\Gamma} (\psi^1 T'_{b}(W^{n,1}_i) \partial_{\boldsymbol{n}^1} u^{n,1}_j+\psi^2 T'_b(W^{n,2}_i)\partial_{\boldsymbol{n}^2} u^{n,2}_j)\right| \leq C,$$
		$$\left|\int_{Q_T} T'_b(W^n_i) \nabla \psi\nabla u^n_j\right| \leq C,$$
		since $\psi^\lambda\in C^{\infty}([0,T]\times \overline{\Omega^\lambda})$ for $\lambda=1,2$, $|T'_b|\leq 1$ and, thanks to Lemma~\ref{compreglemma}, $u_j^n \in L^1\big(0,T;W^{1,1}(\Omega)\big)$ and it is $L^1$ on the membrane.
		The other integral can be computed using the Cauchy-Schwarz inequality and considering the cases $\{W^n_i\leq b\}$ and $\{W^n_i>b\}$ in $Q_T$:
		$$\left|\int_{Q_T} \psi T''_b(W^n_i) \nabla W^n_i\nabla u^n_j\right| =\left|\int_{\{W^n_i\leq b\}\cup \{W^n_i>b\}} \psi T''_b(W^n_i) \nabla W^n_i\nabla u^n_j\right|=$$
		$$= \left|\int_{\{W^n_i\leq b\}} \psi T''_b(W^n_i) \nabla W^n_i\nabla u^n_j\right|\leq C \left(\int_{\{W^n_i\leq b\}}|\nabla u^n_j|^2\right)^\frac{1}{2} \left(\int_{\{W^n_i\leq b\}}|\nabla W^n_i|^2\right)^\frac{1}{2},$$
		since $T'_b(\sigma)=0$ for $\sigma>b$, by definition, and so also $T''_b(\sigma)=0.$
		In order to control the second integral in the right-hand side, we can use the lemma (\hyperlink{proofcontrolS2}{see the proof later on}):
		\begin{lemma}\label{controlS2}(\cite{pierre})
			Let $w$ be solution of \eqref{pb21}. Then, for all $b>0$,
			\begin{equation}
			D\int_{\{|w|\leq b\}} |\nabla w|^2 \leq b\left[\int_{Q_T} f +\int_{\Omega} |w_0|\right].
			\end{equation}
		\end{lemma}
		
		\noindent Applying Lemma~\ref{controlS2} and considering \eqref{unibound}, we infer
		$$ \left(\int_{\{W^n_i\leq b\}}|\nabla W^n_i|^2\right)^\frac{1}{2}\leq C.$$
		Concerning the first integral at the right-hand side, we remark that
		$$\left(\int_{\{W^n_i\leq b\}}|\nabla u^n_j|^2\right)^\frac{1}{2} = \left(\int_{\left\{U^n_i\leq \frac{b}{\eta}-\frac{u_i^n}{\eta}\right\}}|\nabla u^n_i|^2\right)^\frac{1}{2}\leq \left(\int_{\left\{u^n_j\leq \frac{b}{\eta}\right\}}|\nabla u^n_i|^2\right)^\frac{1}{2}\leq \frac{b^\frac{1}{2}}{\eta^\frac{1}{2}}C^\frac{1}{2} \;\mbox{ for } i\neq j,$$
		
		$$\left(\int_{\{W^n_j\leq b\}}|\nabla u_j^n|^2\right)^\frac{1}{2} = \left(\int_{\left\{u_j^n\leq b-\eta U^n_j\right\}}|\nabla u_j^n|^2\right)^\frac{1}{2}\leq \left(\int_{\left\{u_j^n\leq b\right\}}|\nabla u^n_j|^2\right)^\frac{1}{2}\leq (b C)^\frac{1}{2}.$$
		This concludes the proof of Lemma~\ref{controlS}.
		
	\end{proof}
	We now turn to the \hypertarget{proofcontrolS2}{{\bf proof of Lemma~\ref{controlS2}}}.
	\begin{proof}
		We multiply the Equation~\eqref{pb21} by a truncation (non regularized) function $T_b(w)$ and integrate over $Q_T$ to obtain
		$$\int_{Q_T} T_b(w)\partial_t w- \int_{Q_T} D T_b(w) \Delta w = \int_{Q_T} T_b(w) f,$$
		
		$$\int_{\Omega} \int_{w_0}^{w(T)}T_b(w) dw -  \int_0^T\int_{\Gamma}D [T_b(w^1)\partial_{\boldsymbol{n}^1}w^1 + T_b(w^2)\partial_{\boldsymbol{n}^2}w^2]+ \int_{Q_T} D T'_b(w) |\nabla w|^2 = \int_{Q_T} T_b(w) f.$$
		
		\noindent We denote the antiderivative of $T_b$ as $\mathcal{T}(\sigma)=\int_{0}^{\sigma} T_b(s) ds$. So, we compute
		\begin{align*}
		&\int_{\Omega} \int_{w_0}^{w(T)}T_b(w) dw = \int_{\Omega} \mathcal{T}\big(w(T)\big) - \int_{\Omega} \mathcal{T}(w_0),\\[1ex]
		&- \int_0^T\int_{\Gamma}D [T_b(w^1)\partial_{\boldsymbol{n}^1}w^1 + T_b(w^2)\partial_{\boldsymbol{n}^2}w^2] =  \int_0^T\int_{\Gamma}D\; k\; (w^2-w^1)(T_b(w^2)- T_b(w^1)) \geq 0.
		\end{align*}
		
		\noindent Since $\int_{\Omega} \mathcal{T}\big(w(T)\big) \geq 0$ and $T_b(w) \leq b$, we deduce 
		$$D\int_{\{|w|\leq b\}} |\nabla w|^2 \leq b\left[\int_{Q_T} f +\int_{\Omega} |w_0|\right].$$
		This concludes the proof of Lemma~\ref{controlS2}.
		
	\end{proof}
	\subsection{Global existence of a weak solution}\label{step4}
	We conclude the proof of Theorem~\ref{thmprinc}. As before, we consider the approximate system as built in Subsection~\ref{regpb}. Following the previous Theorem~\ref{thmsupersol}, we prove that the supersolution \eqref{supersol} is also a subsolution and, then, a solution of our System \eqref{model}. 
	\begin{thm}\label{thmsol}
			We consider System \eqref{model} together with the conditions on the reaction term \eqref{hpF0}-\eqref{hpF4}
			and $\boldsymbol{u}_0 \in (L^1(\Omega)^+ \cap ({\bf H^1})^*)^m.$ Moreover, we take $k_1=...=k_m$.	
			Then, System \eqref{model} has a weak solution on $(0,+\infty).$
	\end{thm}
	\begin{proof}
		By Theorem~\ref{thmsupersol}, up to a sub-sequence, the approximate solution $\boldsymbol{u}^n $ converges to a weak supersolution. Let us prove that it is also a weak subsolution. We recall some results obtained before:
		$$
		\left\{\begin{array}{l l l}
		\boldsymbol{u}^n \rightarrow \boldsymbol{u}, & \mbox{ in } L^1\big(0,T;L^{\gamma_1}(\Omega))^m, & \forall \gamma_1\in \Big[1,\frac{d}{d-2}\Big),  \\[1ex]
		\nabla \boldsymbol{u}^n \rightharpoonup \nabla \boldsymbol{u}, & \mbox{ in } [L^{\beta}(Q_T)^d]^m, & \forall \beta \in \left[1,\frac{d}{d-1}\right),\\[1ex]
		Tr_\Gamma(\boldsymbol{u}^n) \rightarrow Tr_\Gamma(\boldsymbol{u}), & \mbox{ in } L^1\big(0,T;L^{\beta}(\Gamma)\big)^m, & \forall \beta \in \left[1,\frac{d}{d-1}\right),
		\end{array}
		\right.
		$$
		where for $i=1,...,m$, $f_i(\boldsymbol{u})\in L^1(Q_T)$ and $\forall \psi\in\mathcal{D}_i$, we have \eqref{supersol}.
		We introduce the following notations:
		$$\begin{array}{l l l}
		W^n=\displaystyle\sum_{1\leq i\leq m} u^n_i,\;\;  &Z^n = \displaystyle\sum_{1\leq i\leq m} D_i u^n_i,\;\;  &V^n =\displaystyle\sum_{1\leq i\leq m} D_i k_i(u^{n,2}_i-u^{n,1}_i),\\[3ex]
		W=\displaystyle\sum_{1\leq i\leq m} u_i, &Z=\displaystyle\sum_{1\leq i\leq m} D_i u_i, &V =\displaystyle\sum_{1\leq i\leq m} D_i k_i(u^2_i-u^1_i).
		\end{array}$$
		Adding up the  equations for $u^n_i$, for $i=1,...,m$, in the weak form, we deduce
		\begin{equation}\label{sommasol}
		-\int_{\Omega}
		\psi(0)W^n_{0}+\int_{Q_T} (-\psi_t W^n+\nabla\psi\nabla Z^n)+\int_0^T\int_{\Gamma} [V^n][\psi] = \int_{Q_T} \psi \displaystyle\sum\limits_{1\leq i\leq m} f^n_i.  
		\end{equation}
		Since we have assumed \eqref{hpF1}, -\(\sum_{1\leq i\leq m} f_i^n + C(1+W^n)\geq 0\), with  $\boldsymbol{f}^n(\boldsymbol{u}^n)\rightarrow \boldsymbol{f}(\boldsymbol{u})$ a.e. in $Q_T$ and $W^n$ converges in $L^1(Q_T)$. Applying Fatou's lemma on $-\sum\limits_{1\leq i\leq m} f_i^n + C(1+W^n)\geq 0$, we infer
		$$\int_{Q_T} -\psi\sum\limits_{1\leq i\leq m} f_i(\boldsymbol{u})\leq \liminf_{n\rightarrow+\infty} \int_{Q_T} -\psi\sum\limits_{1\leq i\leq m} f_i^n(\boldsymbol{u}^n).$$
		By a.e convergence of all functions, by $L^1(Q_T)$-convergence of $W^n$ and by Fatou's lemma, we have at the limit for \eqref{sommasol} that
		$$ -\int_{\Omega}
		\psi(0)W_{0}+\int_{Q_T} (-\psi_t W+\nabla\psi\nabla Z)+\int_0^T\int_{\Gamma} [V][\psi] \leq \int_{Q_T} \psi \sum\limits_{1\leq i\leq m} f_i.$$
		Consequently, $W$ is not only a supersolution but also a subsolution. This means that the sum $W$ is a solution and, since its addends $u_i$ are weak supersolutions by Theorem~\ref{thmsupersol}, $\boldsymbol{u}$ is a global weak solution and the proof is completed.
	\end{proof}
	
	Finally, following all the four steps of the proof (from Subsection~\ref{regpb} to Subsection~\ref{step4}), we have proved Theorem~\ref{thmprinc} in the case of interest with quadratic nonlinearities. We point out that this result, as well as Theorem~\ref{thmsupersol}~and~\ref{thmsol}, needs the restricted assumption $k_1=...=k_m$, since it arises in Subsection~\ref{step2}. As said before, we leave as an open problem to remove this restriction. It would also be interesting to see if the method in \cite {CDF2014} can be applied to nearly constant membrane coefficients rather than to the diffusion coefficients. Another open problem, previously introduced, concerns the geometry of the domain. In fact, as we can see in \cite{BCF,LiWang, wang}, we could consider the membrane as the boundary of the domain $\Omega^2$ which is included in $\Omega^1= \Omega \setminus \Omega^2$.

	\section{Declarations	}
	
	{\bf Funding.}  
	The authors have received funding from the European Research Council (ERC) under the European Union's Horizon 2020 research and innovation programme (grant agreement No 740623). The work of G.C. was also partially supported by GNAMPA-INdAM.
	\\[1mm]
	{\bf Conflict of Interest/Declaration of Competing Interest.}
	Not applicable.
	\\[1mm]
	{\bf Availability of data and material.} 
	Not applicable. 
	\\[1mm]
	{\bf Code availability.} 
	Not applicable. 
	\\[1mm]
	{\bf Authors' Contributions.} 
	The two authors have equal contributions.

	\appendix 
	
	\section{Regularity} 
	\label{appendix_reg}
	
	We now analyse in detail  regularity in our problem referring to Lemma~\ref{compreglemma1} that we have rewritten here below, whereas in the next Appendix, we discuss about compactness. We extend previous results for reaction-diffusion systems without membrane \cite{baras, bothe, lp, lpi, pierre} and we refer to~\cite{QuittnerSouplet} for the general theory of parabolic equations.  We also refer to \cite{lpi} for a regularity lemma.
	
	\begin{lemma} [A priori bounds] \label{compreglemma}
		We consider $w$ solution of the following problem in dimension $d\geq 2$
		\begin{equation}\label{pb2}
		\left\{\begin{array}{ll}
		\partial_t w-D\Delta w=f,& \mbox{in}\; Q_T,\\[1ex]
		w=0,& \mbox{in}\; \Sigma_T,\\[1ex]
		\partial_{\boldsymbol{n}^{1}} w^{1}=\partial_{\boldsymbol{n}^{1}} w^2=k(w^2-w^{1}), & \mbox{in}\; \Sigma_{T,\Gamma},\\[1ex]
		w(0,x)=w_0(x)\geq 0, & \mbox{in}\; \Omega,
		\end{array}
		\right.  
		\end{equation}
		with $f\in L^1(Q_T)$ and $w_0\in L^1(\Omega)$. Then, 
		\begin{itemize}
			\item $w \in L^{\beta}\big(0,T;W^{1,\beta}(\Omega)\big), \; \forall \beta \in \left[1,\frac{d}{d-1}\right)$ and $(1+|w|)^\alpha \in L^2\big(0,T;H^1(\Omega)\big)\; \mbox{ for } \alpha \in \left[0,\frac{1}{2}\right)$.
			\item The mapping $(w_0,f)\longmapsto w$ is compact from $L^1(\Omega)\times L^1(Q_T)$ into $L^1\big(0,T;L^{\gamma_1} (\Omega)\big)$, for all  $\gamma_1<\frac{d}{d-2}$ and $L^{\gamma_2}(Q_T)$ for all   ${\gamma_2}  < \frac{2+d}{d}$.
			\item The trace mapping $(w_0,f)\longmapsto Tr_\Gamma(w)\in L^\beta\big(0,T;L^\beta(\Gamma)\big),\; \beta \in \left[1,\frac{d}{d-1}\right)$ is also compact.
			
		\end{itemize}
	\end{lemma}
	
	Notice that we do not use the information $w\in L^2(Q_T)$ here but $w\in L^\infty(0,T; L^1(\Omega))$. That is used in \cite{PierreR} and leads to the exponent $\beta< \frac 4 3$.
	
	\begin{proof}
		The proof is based on manipulating nonlinear quantities and Sobolev imbeddings.  We divide it in several steps.
		\\
		\\
		{\bf Some  $L^2$ regularity of $\nabla w$.}
		Multiplying the equation of $w$ in \eqref{pb2} by $\frac{w}{(1+|w|^\frac{1}{\mu})^\mu}$ and integrating on $\Omega$, we obtain three terms which we estimate separately.
		\\
		\\
		We begin with the Laplacian term. Recalling the membrane conditions and applying the Leibniz rule and the divergence theorem, arguing by a regularization and a limit technique, we gain, since $\frac{w}{(1+|w|^\frac{1}{\mu})^\mu}$ is an increasing function,
		\begin{align*}
		\int_{\Omega} \frac{w}{(1+|w|^\frac{1}{\mu})^{\mu}}\; \Delta w\;\nonumber
		&=\;\int_{\Gamma}\frac{w^1}{(1+|w^1|^{\frac{1}{\mu}})^{\mu}} \partial_{n_1} w^1\;+\;\int_{\Gamma}\frac{w^2}{(1+|w^2|^{\frac{1}{\mu}})^{\mu}} \partial_{n_2} w^2-\;\int_{\Omega} \frac{|\nabla w|^2}{(1+|w|^{\frac{1}{\mu}})^{\mu+1}}\nonumber
		\\
		&=\int_{\Gamma}\left(\frac{w^1}{(1+|w^1|{^\frac{1}{\mu})^{\mu}}} - \frac{w^2}{(1+|w^2|{^\frac{1}{\mu}})^{\mu}}\right) k(w^2-w^1)-\;\int_{\Omega} \frac{|\nabla w|^2}{(1+|w|^{\frac{1}{\mu}})^{\mu+1}}
		\\
		&\leq - \int_{\Omega} \frac{|\nabla w|^2}{(1+|w|^{\frac{1}{\mu}})^{\mu+1}}.
		\end{align*}		
		
		We analyse now the reaction term. We remark that $0\leq \frac{w}{(1+|w|^\frac{1}{\mu})^\mu}\leq 1$ and, using that $f\in L^1(Q_T)$, we conclude 
		$$
		\int_{\Omega} \left|\frac{w}{(1+|w|^\frac{1}{\mu})^{\mu}}\;f\right| \leq \int_{\Omega} |f| = \|f\|_{L^1(\Omega)}.
		$$
		
		Next, for the time derivative, we define  the anti-derivative $0 \leq \psi_\mu (w)= \int_0^w \frac{v\, dv}{(1+|v| ^\frac{1}{\mu})^{\mu}} \leq w$, then 
		$$
		\frac{w}{(1+|w|^\frac{1}{\mu})^{\mu}}\;\partial_t w\;=:\; \partial_t \psi_{\mu}(w).
		$$
		Therefore, combining the previous equality and inequalities,  we find
		$$
		\int_{\Omega} \partial_t \psi_{\mu}(w)\;+\;D\int_{\Omega}  \frac{|\nabla w|^2}{(1+|w|^{\frac{1}{\mu}})^{\mu+1}}\leq \|f\|_{L^1(\Omega)}.
		$$
		At this point, we can integrate in time and  obtain
		$$
		D\int_{Q_T}  \frac{|\nabla w|^2}{(1+|w|^{\frac{1}{\mu}})^{\mu+1}} \leq \int_{\Omega} \psi_{\mu}\big(w_0(x)\big) + \|f\|_{L^1(Q_T)} \leq 
		\|w_0\|_{L^1(\Omega)} + \|f\|_{L^1(Q_T)}.
		$$
		Since, for all $\mu > 1$ there is a $C_\mu$ such that 
		$$
		(1+|w|^{\frac{1}{\mu}})^{\mu+1} \leq C_\mu (1+|w|)^{2(1-\alpha)}, \qquad  \alpha= \frac{1}{2}\left(1-\frac{1}{\mu}\right),
		$$
		we conclude that 
		$$
		\int_{Q_T} (1+|w|)^{2(\alpha -1)} |\nabla w|^2  \leq \frac{C_\mu}{D}\left[ \|w_0\|_{L^1(\Omega)} + \|f\|_{L^1(Q_T)}\right] , \qquad 0 < \alpha < \frac 12.
		$$
		And thus, there is a constant $C_\alpha$ which also depends on $\|w_0\|_{L^1(\Omega)} + \|f\|_{L^1(Q_T)}$ such that 
		\begin{equation}\label{phi1}
		\int_{Q_T} |\nabla (1+|w|)^\alpha |^2  \leq C_\alpha ,  \qquad 0 < \alpha < \frac 12.
		\end{equation}
		{\bf Integrability of $w$.}
		The Sobolev imbedding (see Appendix~\ref{appendix_inequal} ) gives
		\begin{equation}\label{phi2}
		\left(\int_\Omega (1+|w|)^{\alpha 2^*} \right)^{\frac{2}{2^*}} \leq C\int_{\Omega} |\nabla (1+|w|)^\alpha |^2, \quad \qquad  {2^*}= \frac {2d}{d-2}.
		\end{equation}
		which is only useful when $\alpha 2^* >1$, i.e. $ \frac{d-2}{2d} <\alpha$. Then, we can interpolate between  $L^1$ and $L^{\alpha 2^*}$ and find
		$$
		\left(\int_\Omega (1+|w|)^{\gamma} \right)^{\frac{1}{\gamma}} \leq C \left(\int_\Omega (1+|w|)\right)^{\theta}  \left( \int_{\Omega} |\nabla (1+|w|)^\alpha |^2\right)^{\frac{1-\theta}{2\alpha}}, \qquad \frac 1 \gamma = \theta +\frac{1-\theta}{\alpha 2^*} .
		$$
		We may choose $\frac{1-\theta}{2\alpha}=1$,   and, recalling that $\alpha < \frac 12$, we find the integrability 
		\begin{equation*}\label{gamma1}
		w\in L^1\big(0,T; L^{\gamma_1}(\Omega)\big) \qquad \text{with} \qquad {\gamma_1} = \frac{d}{2\big( d(1-\alpha) -1\big) } < \frac{d}{d-2} .
		\end{equation*}
		We may also choose $\frac{\gamma (1-\theta)}{2\alpha}=1$, $\alpha < \frac 12$ and find the integrability
		\begin{equation*}\label{gamma2}
		w\in L^{\gamma_2}(Q_T)  \qquad \text{with} \qquad {\gamma_2} = \frac{2\big(1+\alpha d\big)}{d} < \frac{2+d}{d} .
		\end{equation*}
		\\
		{\bf Regularity of $\nabla w$.}
		On the other hand, Hölder inequality gives 
		\[  \begin{array}{rl}
		\displaystyle   \int_{\Omega} |\nabla w|^\beta  =\int_{\Omega} \frac{|\nabla w|^\beta}{(1+|w|)^\eta}(1+|w|)^\eta
		&  \displaystyle \leq  \left(\int_{\Omega} \frac{|\nabla w|^{\beta r}}{(1+|w|)^{\eta r}}\right)^\frac{1}{r}\left(\int_{\Omega}(1+|w|)^{\eta p}\right)^\frac{1}{p} 
		\\[15pt]
		&  \displaystyle\leq C\left( \int_{\Omega}   |\nabla  (1+|w|)^\alpha  |^2  \right)^{\frac 1 r} \left(\int_{\Omega}(1+|w|)^{\eta p}\right)^\frac{1}{p} 
		\end{array}  \]
		with
		$$
		\frac{1}{r}+\frac{1}{p}=1, \qquad   \beta=\frac{2}{r} \leq 2, \quad \eta r=2 (1-\alpha). 
		$$
		
		We can choose $\eta p = \gamma_1$ from above, which requires $\eta \left( \frac{1}{2(1-\alpha)} + \frac 1{\gamma_1} \right)= 1$, $\beta = \frac \eta{1-\alpha}= \frac{2 \gamma_1}{\gamma_1+2(1-\alpha)}$
		and we find, thanks to the estimate \eqref{phi1}, 
		\[
		\int_{\Omega} |\nabla w|^\beta  \in L^1(0,T)   \qquad \text{with} \qquad   \beta <  \frac d {d-1}.
		\]
		This concludes the proof of the  gradient estimate. Moreover, considering that $\beta<\gamma_2$, thanks to Sobolev imbeddings, we can infer that $w \in L^\beta(0,T;L^\beta(\Omega))$. 
		\\
		\\
		{\bf The trace.}  The regularity of the trace derives from its continuity property \cite{brezis} (p.$315$), i.e. 
		\begin{equation}
		\int_0^T\|\mbox{Tr}(w)\|^\beta_{W^{1-\frac{1}{\beta},\beta}(\Gamma)}\leq \int_0^T\|w\|^\beta_{W^{1,\beta}(\Omega)}, \quad 1\leq \beta<\frac{d}{d-1}.
		\end{equation}

		\section{Compactness}
		\label{appendix_compact}
		In order to  conclude the proof of Lemma~\ref{compreglemma}, it remains to adapt compactness arguments to the case of the membrane problem. A proof based  on  a dual approach, see \cite{baras, bothe}, could be used. We rather go to a direct proof. \\ [1ex]
		
		\noindent{\bf Compactness  in space.}   It  can be obtained using the Rellich-Kondrachov theorem \cite{adams}, since we know the approximate family is bounded in the spaces $W^{1,\beta}(\Omega^\lambda)$, $\lambda=1, \; 2$ which are compactly embedded in $L^{\gamma_1}(\Omega^\lambda)$, with $\gamma_1<\frac{d}{d-2}$.
		\\ [1ex]
		
		\noindent{\bf Compactness  in time.} We use  the  Fréchet-Kolmogorov criteria, see~\cite{brezis} for instance. Let $\varphi(x)$ be a nonnegative, radially symmetric,  $C^\infty_c(\R^d)$ standard mollifier with mass $1$. We define the family $(\varphi_\delta)_{\delta>0}$ by
		\begin{equation} \label{mollif3}
		\varphi_\delta(x)=\frac{1}{\delta^d}\,\varphi\left(\frac{x}{\delta}\right), \qquad \|\varphi_\delta\|_{L^1(\Omega)}=1.
		\end{equation}
		Moreover, we have
		\begin{equation}\label{mollif2}
		\| g\ast \varphi_\delta \|_{L^p(\Omega)}\leq \| \varphi_\delta \|_{L^1(\Omega)}\|g\|_{L^p(\Omega)}     ,   	
		\end{equation}
		and it holds (\cite{evans}, p.$273$) that for any function $g\in W^{1,p}(\Omega)$,
		\begin{equation}\label{mollif1}
		\| g\ast \varphi_\delta -g\|_{L^p(\Omega)}\leq \delta \|\nabla g\|_{L^p(\Omega)}.
		\end{equation}
		About the derivative of order $k$ of $\varphi_\delta$, we know that
		\begin{equation}\label{mollif4}
		\nabla^k \varphi_\delta(x)=\frac{1}{\delta^{d+k}}\nabla^k \varphi\left(\frac{x}{\delta}\right), \qquad \|\nabla^k \varphi_\delta\|_{L^1(\Omega)}\leq \frac{C}{\delta^{k}}.
		\end{equation}

		\begin{proof}
			To complete the proof of time compactness, we shall prove that, as $h \rightarrow 0$, 
			\begin{equation}\label{eq:Kolm1}
			\ithx |w(t+h,x)-w(t,x)|dx dt \rightarrow 0.
			\end{equation}
			
			By comparison with the mollified versions, the triangular equality yields
			\begin{align*}
			\ithx |w(t+h,x)-w(t,x)|dx dt &\leq \ithx |w(t,x)- w(t,\cdot)\ast\varphi_\delta(x)|dx dt\nonumber \\ 
			&+ \ithx |w(t+h,x)- w(t+h,\cdot)\ast\varphi_\delta(x)|dx dt \nonumber \\ &+\ithx |w(t+h,\cdot)\ast\varphi_\delta(x)-w(t,\cdot)\ast\varphi_\delta(x)|dx dt
			\end{align*} 
			Here, $\delta$ depends on $h$ (to be specified later on) and converges to zero.
			It suffices to prove that each integral converges to zero as $h\rightarrow 0$.
			\\[1ex]
			
			\noindent{\it First term.}
			We analyse the first term in the right-hand side. It holds that
			\begin{equation}\label{eq:Kolm2}
			\ithx |w(t,x)- w(t,\cdot)\ast\varphi_\delta(x)|dx dt\leq \delta \int_0^{T-h} \|\nabla w(t,x)\|_{L^1(\Omega)} dt \leq C\delta(h),
			\end{equation}
			thanks to $w$ regularity and to \eqref{mollif1}, which proves that it converges to zero as $h\rightarrow 0$.\\[1ex]	
			{\it Second term.}
			For the second integral, we can proceed as for the fist one obtaining
			\begin{equation}\label{eq:Kolm3}
			\ithx |w(t+h,x)- w(t+h,\cdot)\ast\varphi_\delta(x)|dx dt \leq C \delta(h).
			\end{equation}
			
			\noindent{\it Third term.} 
			Remembering \eqref{pb2}, the last term can be written as
			\begin{equation*}
			\ithx |w(t+h,\cdot)\ast\varphi_\delta(x)-w(t,\cdot)\ast\varphi_\delta(x)|dx\, dt=\ithx \left|\int_{t}^{t+h} \frac{\partial w}{\partial s}(s,x)\ast \varphi_\delta(x) ds\right| dx\,dt
			\end{equation*}
			\begin{equation*}
			=\ithx \left|\int_{t}^{t+h} \left[D\Delta w + f\right]\ast \varphi_\delta ds\right| dx\,dt = \ithx \left|\int_{t}^{t+h} D w \ast \Delta\varphi_\delta+ f\ast \varphi_\delta\, ds \right| dx\,dt 
			\end{equation*}
			after exchanging derivatives in the convolution.
			From \eqref{mollif2} we deduce
			\begin{align*}
			\ithx |w(t+h,\cdot)\ast\varphi_\delta(x)-w(t,\cdot)\ast\varphi_\delta(x)|dx\, dt &\leq \int_{0}^{T-h}\int_t^{t+h} D\|w\|_{L^1(\Omega)} \|\Delta \varphi_\delta\|_{L^1(\Omega)}\nonumber\\
			&+ \int_{0}^{T-h}\int_t^{t+h}\|f\|_{L^1(\Omega)} \|\varphi_\delta\|_{L^1(\Omega)}.
			\end{align*}
			Finally, thanks to \eqref{mollif3} and \eqref{mollif4}, we obtain choosing $\delta=h^{1/4}$  
			\[
			\ithx |w(t+h,\cdot)\ast\varphi_\delta(x)-w(t,\cdot)\ast\varphi_\delta(x)|dx\, dt \leq C[ \frac{h}{\delta^2} +h]  \leq C \sqrt h 
			\]	
			and  \eqref{eq:Kolm1} follows combining this estimate with \eqref{eq:Kolm2} and \eqref{eq:Kolm3}.
		\end{proof}
		
		Applying the Fréchet-Kolmogorov theorem \cite{brezis}, we conclude that the set  of functions $w \in L^1(Q_T)$ under consideration is compact in  $L^1(Q_T)$.  Consequently, we claim compactness in $L^1\big(0,T; L^{\gamma_1}(\Omega)\big)$ with ${\gamma_1}< \frac{d}{d-2}$ and in $L^{\gamma_2}(Q_T)$ with ${\gamma_2}< \frac{2+d}{d} $.
		In fact, since we have $L^1$-convergence of $L^p$-functions, we deduce convergence in the space $L^q$, for $q<p$.
		\\
		
		\noindent{\bf Compactness of traces in $L^\beta\big(0,T;L^\beta(\Gamma)\big)$.} Space compactness can be deduced, in each $\Omega^\lambda$,  from trace continuity and a compactness result for the boundary (\cite{demengel}, p.167) such that $W^{1-\frac{1}{\beta},\beta}(\Gamma) \subset\subset L^{\beta}(\Gamma).$ Time compactness is again achieved through the Fréchet-Kolmogorov theorem. Following the same proof as before and changing the order of the time integrals, we need to recall Kedem-Katchalsky membrane conditions from which we can infer that $\pt Tr_\Gamma (w) \in L^1(0,T;L^1(\Gamma))$ and so we can conclude the proof.

	\end{proof}
	\section{Sobolev and Poincaré inequalities with membrane}
	\label{appendix_inequal}
	
	For completeness, we explain why  the Sobolev embeddings  can be extended to the  membrane problem, leading to \eqref{phi1} and \eqref{phi2}. More precisely, we explain how to arrive to 
	$$ \| \phi_{\alpha}(w^1)\|^2_{L^{2^*}(\Omega^1)} + \| \phi_{\alpha}(w^2)\|^2_{L^{2^*}(\Omega^2)}\;  \leq C \left( \|\nabla \phi_{\alpha}(w^1)\|^2_{L^2(\Omega^1)} + \|\nabla \phi_{\alpha}(w^2)\|^2_{L^2(\Omega^2)} \right).$$
	There are two difficulties. First, the boundary condition is not Dirichlet everywhere. Second we are dealing with a singular domain $\Omega$ and so we cannot use directly  the Sobolev or Poincaré inequalities in $\Omega$, but only some easy generalizations that we explain now.
	\\
	
	We are going to prove the	
	\begin{thm}[Gagliardo-Nirenberg-Sobolev inequality  with membrane]\label{problem}
		We consider the bounded domain $\Omega = \Omega^1\, \cup \,\Omega^2 \subset \R^d, \; d\geq 2$, with piecewise $C^1$ sub-domains $\Omega^1$ and $\Omega^2$ and a $C^1$ membrane $\Gamma=\partial\Omega^1\, \cap \,\partial\Omega^2$ which decomposes $\Omega$ in the two parts. We take the function $v=(v^1,v^2) \in {\bf H^1}$ (see Definition \ref{defh1}), then, for $\lambda=1,2$, 
		\begin{equation}\label{pb0}
		\|v^\lambda\|_{L^{2^*}(\Omega^\lambda)}\;\leq\;C(\Omega^\lambda)\;\;\|\nabla v^\lambda\|_{{L^{2}(\Omega^\lambda)}^d},
		\end{equation}
		and consequently 
		\begin{equation}\label{pb1}
		[\;\|v^1\|_{L^{2^*}(\Omega^1)}\;+\;\|v^2\|_{L^{2^*}(\Omega^2)}\;]\;\leq\;C(\Omega^1,\Omega^2)\;[\;\|\nabla v^1\|_{{L^{2}(\Omega^1)}^d}\;+\;\|\nabla v^2\|_{{L^{2}(\Omega^2)}^d}\;].
		\end{equation}
	\end{thm}
	The reason why we want to prove this theorem is that the domain $\Omega$ described above is not enough regular to use the usual Gagliardo-Nirenberg-Sobolev inequality (\cite{brezis}, p.284). Consequently, we need to build  smoother domains containing each $\Omega^\lambda$, $\lambda=1,2$, in which we can apply known results and then, with a restriction to $\Omega$, we can find \eqref{pb0} and \eqref{pb1}. The construction is made considering an extension of $\Gamma$ and a domain with the same internal structure as $\Omega$ such that it contains $\Omega$ and each extension of the $\Omega^\lambda$ is of class $C^1$.

	We first recall the standard Sobolev inequality (\cite{brezis}, p.$284$) in a bounded open set.
	\begin{thm}[Sobolev embedding]\label{h1}
		Let $Q$ be a bounded  open subset of class $C^1$ in $\R^d$. There is  a constant   $  C_Q$  such that for all $v\in H^1(Q),$ we have
		$$v\in L^{2^*}(Q) \quad\mbox{ and }\quad \|v\|_{L^{2^*}(Q)}\;\leq\;C_Q\;\left[\;\|v\|_{L^{2}(Q)}\;+\;\|\nabla v\|_{{L^{2}(Q)}^d}\;\right].$$
	\end{thm}
	
	\begin{proof}
		We recall how to  prove Theorem~\ref{h1} departing  from the case of the full space. We use the regularity of the domain which assures us the existence of a linear and continuous extension operator $T: H^1(Q) \rightarrow H^1(R^d)$, which is also the extension from $L^2(Q)$ into $L^2(\R^d)$ (\cite{brezis}, p.$272$). So, we obtain that:
		\begin{eqnarray}
		&&\bullet\; \mbox{ taken } v\in H^1(Q), \quad T(v)\in H^1(\R^d) \mbox{ and } T(v)=v \mbox{ on } Q;\label{T1}\\[1ex]
		&&\bullet\; \|T(v)\|^2_{L^2(\R^d)} \leq C^2_{\footnotesize{\mbox{exten}}L^2}(Q)\;\|v\|^2_{L^2(Q)};\label{T2}\\[1ex]
		&&\bullet\; \|\nabla T(v)\|^2_{{L^2(\R^d)}^d} \leq C^2_{\footnotesize{\mbox{exten}}H^1}(Q)\;\|v\|^2_{H^1(Q)}.\label{T3}
		\end{eqnarray}
		Moreover, for construction (see the proof of the extension theorem \cite{brezis}, p.272), this operator is in $H^1_0(R^d)$. Consequently, using a corollary of the Sobolev inequality (\cite{evans}, p.265), we get that
		$$T(v)\in L^{2^*}(\R^d) \mbox{ and } \|T(v)\|_{L^{2^*}(\R^d)}\leq C_{\footnotesize{\mbox{sob}}}(d,2)\; \|\nabla T(v)\|_{{L^2(\R^d)}^d}.$$
		We proceed with some estimates due to the application of \eqref{T1}, \eqref{T2}, \eqref{T3}.
		First of all, we deduce
		$$\|\nabla v\|^2_{{L^2(Q)}^d}=\|\nabla T(v)\|_{{L^2(Q)}^d}\leq \|\nabla T(v)\|_{{L^2(\R^d)}^d}\leq C_{\footnotesize{\mbox{exten}}H^1}(Q)\;\nh{v}{Q}^2$$$$=C_{\footnotesize{\mbox{exten}}H^1}(Q)\;\left[\nl{v}{Q}+\nlnab{v}{Q}\right].$$
		Since $T(v)\in L^{2^*}(\R^d)$ and $T(v)=v$ on $Q$, we get $v\in L^{2^*}(Q)$ and
		$$\nlstar{v}{*}{Q}^2= \nlstar{T(v)}{*}{Q}^2\leq \nlstar{T(v)}{*}{\R^d}^2\leq (C_{\footnotesize{\mbox{sob}}}(d,2))^2 \;\nlnab{T(v)}{\R^d}^2$$$$\leq(C_{\footnotesize{\mbox{sob}}}(d,2))^2\;C^2_{\footnotesize{\mbox{exten}}H^1}(Q)\;\left[\nl{v}{Q}^2+\nlnab{v}{U}^2\right].$$
		The proof of Theorem~\ref{h1} is complete.
	\end{proof}	
	
	Since we do not impose Dirichlet conditions on the full boundary, we need the following generalized Poincaré inequality (\cite{morrey} p.82).
	\begin{thm}[Poincaré inequality]\label{poin}
			Suppose $Q$ a bounded and connected open subset of $\R^d$ of class $C^1$ and consider a portion of its boundary $\Sigma_0 \subset \partial Q $ such that $|\Sigma_0|>0$. Then, there exists a constant $C(Q,\Sigma_0)$ such that
			\begin{equation}
			\forall v\in H^1(Q) \mbox{ such that } Tr_{\Sigma_0}(v)=0, \quad \nl{v}{Q}^2\leq C(Q,\Sigma_0) \nlnab{v}{Q}^2.
			\end{equation}
	\end{thm}
	\begin{proof}
		If the statement is not true, we can find a sequence ${v_n}$ such that each $v_n\in H^1(Q)$ and
		$$\nl{v_n}{Q}^2>n\;\left[\nlnab{v_n}{Q}^2\;+\;\left(\int_{\Sigma_0}|v_n|dS\right)^2\right].$$
		On account of the homogeneity (normalizing), we may assume that $\nl{v_n}{Q}=1$, for each $n$. So we infer that 
		\begin{equation}\label{ineq}
		n\;\left[\nlnab{v_n}{Q}^2\;+\;\left(\int_{\Sigma_0}|v_n|dS\right)^2\right]<1,
		\end{equation}
		which implies that 
		$$\nlnab{v_n}{Q}^2<\frac{1}{n}.$$
		Therefore, $\nabla v_n \rightarrow 0$ in $L^2(Q)$.
		Moreover, $v_n$ is bounded in $H^1(Q)$, so, up to a sub-sequence, it converges weakly in $H^1(Q)$ to some $v$. So $\nabla v_n \rightharpoonup \nabla v$, that means $\nabla v=0.$ This shows that $v$ is a constant (since $Q$ is connected). For the continuity of the trace operator and \eqref{ineq}, we deduce
		$$0=\lim_{n\rightarrow +\infty}\int_{\Sigma_0} |v_n| dS=\int_{\Gamma_0} |v| dS=|c||\Gamma_0|,$$
		and so $v=0$.\\
		At the same time, thanks to the Rellich-Kondrachov compactness theorem \cite{adams, brezis, evans}, up to a sub-sequence, $v_n$ converges strongly in $L^2(Q)$ to $v=0$. Hence, since $\nl{v_n}{Q}=1$, we arrive to a contradiction.
		
	\end{proof}
	
	At this point we are able to give the proof of Theorem~\ref{problem}.
	\begin{proof}
		We apply Theorems~\ref{h1}~and~\ref{poin}. First of all we consider the extension of $\Gamma$ into the space $\R^d$ such that now $\Gamma$ separates the space into two pieces $P^\lambda$ with $\lambda=1,2$. Since we have Dirichlet boundary conditions on $\Gamma^\lambda$, we can extend the function to zero in the whole $P^\lambda$. So now, considering $Q^\lambda$ a domain of class $C^1$ such that $\Omega^\lambda  \subset Q^\lambda \subset P^\lambda$ and for $\lambda, \sigma=1,2$, $\;Q^\lambda \cap P^\sigma$ is a portion of $\Gamma$, we can apply Theorems~\ref{h1}~and~\ref{poin} to
		$$\tilde{v}^\lambda=\left\{\begin{array}{ll}
		v^\lambda, & \mbox{ in } \Omega^\lambda,   \\
		0, & \mbox{ in } \Gamma^\lambda \cup \{Q^\lambda\setminus \Omega^\lambda\}. 
		\end{array}\right.$$
		This proves Theorem \ref{problem} in $Q^\lambda$ and, so, in $\Omega^\lambda$. 
		
	\end{proof}
	
	\bibliographystyle{acm}
	\bibliography{references}
\end{document}